\documentclass{scrartcl}
\usepackage{amsmath,amssymb,amsthm}
\usepackage{mathrsfs}
\usepackage{color}
\usepackage{ascmac}
\usepackage{dsfont}
\usepackage{xcolor}
\usepackage{mathtools}

\numberwithin{equation}{section} 

\newcommand{\np}[1]{(#1)}
\newcommand{\nb}[1]{[#1]}
\newcommand{\bp}[1]{\big(#1\big)}
\newcommand{\bb}[1]{\big[#1\big]}
\newcommand{\Bp}[1]{\bigg(#1\bigg)}

\newcommand{\Bb}[1]{\bigg[#1\bigg]}
\newcommand{\calm}{{\mathcal M}}

\newcommand{\calp}{{\mathcal P}}

\newcommand{\calt}{{\mathcal T}}

\newcommand{\calw}{{\mathcal W}}

\newcommand{\R}{\mathbb{R}}
\newcommand{\C}{\mathbb{C}}
\newcommand{\Z}{\mathbb{Z}}
\newcommand{\N}{\mathbb{N}}

\DeclareMathOperator{\e}{e}

\DeclareMathOperator{\Div}{div}

\DeclareMathOperator{\supp}{supp}

\newcommand{\set}[1]{\ensuremath{\{#1\}}}
\newcommand{\setl}[1]{\ensuremath{\bigl\{#1\bigr\}}}

\newcommand{\setc}[2]{\ensuremath{\{#1\ \vert\ #2\}}}
\newcommand{\setcl}[2]{\ensuremath{\bigl\{#1\ \big\vert\ #2\bigr\}}}

\newcommand{\ball}{B}
\newcommand{\proj}{\calp}
\newcommand{\projcompl}{\calp_\bot}
\newcommand{\grp}{G}
\newcommand{\dualgrp}{\widehat{G}}
\newcommand{\torus}{{\mathbb T}}
\newcommand{\Rn}{{\R^n}}
\newcommand{\grad}{\nabla}

\newcommand{\dx}{{\mathrm d}x}

\newcommand{\ds}{{\mathrm d}s}
\newcommand{\dt}{{\mathrm d}t}

\newcommand{\dxi}{{\mathrm d}\xi}
\newcommand{\SR}{\mathscr{S}}
\newcommand{\TDR}{\mathscr{S^\prime}}

\newcommand{\FT}{\mathscr{F}}
\newcommand{\iFT}{\mathscr{F}^{-1}}
\newcommand{\norm}[1]{\lVert#1\rVert}
\newcommand{\norml}[1]{\big\lVert#1\big\rVert}
\newcommand{\normL}[1]{\Bigl\lVert#1\Bigr\rVert}
\newcommand{\snorm}[1]{{\lvert #1 \rvert}}
\newcommand{\snorml}[1]{{\bigl\lvert #1 \big\rvert}}

\newcommand{\LR}[1]{\mathrm{L}^{#1}}
\newcommand{\LRloc}[1]{\mathrm{L}^{#1}_{\mathrm{loc}}} 
\newcommand{\CR}[1]{\mathrm{C}^{#1}}  
\newcommand{\CRi}{\CR \infty}
\newcommand{\CRci}{\CR \infty_0}

\newcommand{\WSR}[2]{\mathrm{W}^{#1,#2}} 
 
\newcommand{\DSR}[2]{\mathrm{D}^{#1,#2}}

\newcommand{\CRcisigma}{\CR{\infty}_{0,\sigma}}

\newcommand{\vvel}{v}
\newcommand{\vpres}{p}

\newcommand{\Vvel}{V}
\newcommand{\Vpres}{P}

\newcommand{\wvel}{w}
\newcommand{\wpres}{\mathfrak{q}}
\newcommand{\Wvel}{W}
\newcommand{\Wpres}{\mathfrak{Q}}
\newcommand{\twvel}{\tilde{w}}

\newcommand{\uvel}{u}

\newcommand{\upres}{\mathfrak{p}}
\newcommand{\Uvel}{U}
\newcommand{\Upres}{\mathfrak{P}}
\newcommand{\tuvel}{\tilde{u}}
\newcommand{\tupres}{\tilde{\mathfrak{p}}}

\newcommand{\tin}{\text{in }}
\newcommand{\tif}{\text{if }}
\newcommand{\ton}{\text{on }}

\newcommand{\tfor}{\text{for }}

\newcommand{\half}{\frac{1}{2}}
\newcommand{\rey}{\tau}
\newcommand{\per}{\calt}

\newcommand{\perf}{\frac{2\pi}{\per}}
\newcommand{\perft}{\tfrac{2\pi}{\per}}
\newcommand{\eone}{\e_1}

\theoremstyle{definition}
 \newtheorem{Definition}{Definition}[section]
   
\theoremstyle{plain}
 \newtheorem{Theorem}[Definition]{Theorem}
 \newtheorem{Proposition}[Definition]{Proposition}
 \newtheorem{Lemma}[Definition]{Lemma}

\begin{document}
\title{On the regularity of weak solutions to time-periodic Navier--Stokes equations in exterior domains}

\author{Thomas Eiter%
}

\maketitle

\begin{abstract}
Consider the time-periodic
viscous incompressible fluid flow past a body 
with non-zero velocity at infinity. 
This article gives sufficient conditions 
such that weak solutions to this problem
are smooth.
Since time-periodic solutions 
do not have finite kinetic energy in general,
the well-known regularity results for weak solutions 
to the corresponding 
initial-value problem cannot be transferred directly.
The established regularity criterion
demands a certain integrability of the purely periodic part
of the velocity field or its gradient,
but it does not concern the time mean of these quantities.
\end{abstract}


\section{Introduction}

Consider the time-periodic flow of a viscous incompressible fluid past a three-dimensional body
that translates with constant non-zero velocity $\vvel_\infty$.
We assume $\vvel_\infty$ to be directed along
the $x_1$-axis
such that $\vvel_\infty=\rey\eone$ with $\rey>0$.
In a frame attached to the body, 
the fluid motion 
is then governed by the
Navier--Stokes equations
\begin{subequations}\label{sys:NavierStokesTP}
\begin{alignat}{2}
\partial_t\uvel-\Delta\uvel-\rey\partial_1\uvel+\uvel\cdot\grad\uvel+\grad\upres&=f
&&\qquad\tin\torus\times\Omega, 
\label{sys:NavierStokesTP.a}
\\
\Div\uvel&=0 
&&\qquad\tin\torus\times\Omega, \\
\uvel&=\uvel_\ast 
&&\qquad\ton\torus\times\partial\Omega, \\
\lim_{\snorm{x}\to\infty}\uvel(t,x)&=0 
&&\qquad\tfor t\in\torus,
\end{alignat}
\end{subequations}
where $\Omega\subset\R^3$ is the exterior domain occupied by the fluid.
The functions $\uvel\colon\torus\times\Omega\to\R^3$ and $\upres\colon\torus\times\Omega\to\R$ are velocity and pressure 
of the fluid flow,
$f\colon\torus\times\Omega\to\R^3$ is an external body force,
and $\uvel_\ast\colon\torus\times\partial\Omega\to\R^3$
denotes the velocity field at the boundary.
The time axis is given by the torus group 
$\torus\coloneqq\R/\per\Z$, which 
ensures that all functions appearing in \eqref{sys:NavierStokesTP} 
are time periodic with a prescribed period $\per>0$.

In this article, we study weak solutions to \eqref{sys:NavierStokesTP},
and we provide sufficient conditions 
such that these weak solutions possess more regularity
and are actually smooth solutions.
In the context of the initial-value problem for
the Navier--Stokes equations,
those criteria have been studied extensively. 
Existence of weak solutions 
was shown several decades ago 
in the seminal works by Leray~\cite{Leray_SurLeMouvementLiquidVisqeux_1934}
and Hopf~\cite{Hopf}
together with a corresponding energy inequality,
but it remained unclear for many decades 
whether solutions in this \textit{Leray--Hopf class}
are unique, even when the external forcing 
is smooth (or even $0$).
Note that Albritton, Bru\'e and Colombo~\cite{AlbrittonBrueColombo_NonUniqLeraySolForcedNSE_2022}
recently showed
that there are forcing terms
such that multiple Leray--Hopf solutions
to the initial-value problem exist,
so that uniqueness fails for general forcing terms.
However, 
Leray--Hopf solutions
come along with a weak-strong uniqueness 
principle
that states that weak solutions coincide with strong solutions
if the latter exist.
This also motivated the development of criteria
that ensured higher regularity of weak solutions.
The first results in this direction 
are due to Leray~\cite{Leray_SurLeMouvementLiquidVisqeux_1934},
and Serrin~\cite{Serrin_IntRegWeakSolNSE_1962},
who showed that 
if a weak solution is an element of $\LR{\rho}(0,T;\LR{\kappa}(\Omega)^3)$
for some $\kappa,\rho\in(1,\infty)$ 
such that
$\frac{2}{\rho}+\frac{3}{\kappa}< 1$,
then it is 
a strong solution and smooth with respect to the spatial variables.
Since then, there appeared many other 
regularity criteria 
that ensured higher-order regularity of a weak solution
to the initial-value problem; 
see~\cite{BeiraoDaVeiga_ConcRegProblSolNSE_1995, ChaeChoe_RegSolNSE_1999, SereginSverak_NSELowerBoundsPressure_2002, PenelPokorny_NewRegCritNSEContGradientVel_2004, NeustupaPenel_RegWeakSolNSEOneComonentSpectralProjVorticity_2014, NeustupaYangPenel_RegCritWeakSolNSESpectrProjVortVel_2022}
and the references therein.

To obtain similar regularity results 
for weak solutions to the time-periodic problem~\eqref{sys:NavierStokesTP},
the first idea might be to 
identify these with weak solutions 
to the initial-value problem
for a suitable initial value.
However, this procedure is not successful 
in the considered framework of an exterior domain $\Omega$
since regularity of weak solutions to the initial-value problem 
is usually investigated within the class
$\LR{\infty}(0,T;\LR{2}(\Omega)^3)$,
but weak solutions $\uvel$ to the time-periodic problem
are merely elements of 
$\LR{2}(\torus;\LR{6}(\Omega)^3)$ at the outset; see Definition~\ref{def:WeakSolution_NStp} below.
To see that we cannot expect the same integrability 
as for the initial-value problem,
observe that
every weak solution to the steady-state problem
is also a time-periodic solution.
In general, these steady-state solutions do not have finite kinetic energy 
but only belong to $\LR{s}(\Omega)^3$ for $s>2$;
see Theorem~\ref{thm:statNS} below.
Therefore, 
one cannot reduce the time-periodic situation 
to that of the initial-value problem.

For the formulation of 
suitable regularity criteria 
for time-periodic weak solutions,
we decompose functions into a 
time-independent part, 
given by the time mean over one period,
and a time-periodic remainder part.
To this decomposition, we associate a pair of 
complementary projections $\proj$ and $\projcompl$
such that
\[
\proj\uvel\coloneqq\int_\torus\uvel(t,\cdot)\,\dt,
\qquad
\projcompl\uvel\coloneqq \uvel - \proj\uvel.
\]
Then $\proj\uvel$ is called the \textit{steady-state part} of $\uvel$,
and $\projcompl\uvel$ denotes the \textit{purely periodic part} of $\uvel$.

In this article, we consider weak solutions 
to \eqref{sys:NavierStokesTP}
in the following sense.

\begin{Definition}\label{def:WeakSolution_NStp}
Let $f\in\LRloc{1}(\torus\times\Omega)^3$
and $\uvel_\ast\in\LRloc{1}(\torus\times\partial\Omega)^3$.
A function $\uvel\in\LRloc{1}(\torus\times\Omega)^3$ 
is called \emph{weak solution} to~\eqref{sys:NavierStokesTP}
if is satisfies the following properties:
\begin{enumerate}
\item[i.]
$\grad\uvel\in\LR{2}(\torus\times\Omega)^{3\times3}$, 
$\uvel\in\LR{2}(\torus;\LR{6}(\Omega)^3)$,
$\Div\uvel=0$ in $\torus\times\Omega$,
$\uvel=\uvel_\ast$ on $\torus\times\partial\Omega$,
\label{item:WeakSolution_NStp_L2}
\item[ii.]
$\projcompl\uvel\in\LR{\infty}(\torus;\LR{2}(\Omega))^3$,
\label{item:WeakSolution_NStp_LInfty}
\item[iii.]
the identity
\[
\int_\torus\int_\Omega\bb{-\uvel\cdot\partial_t\varphi
+\grad\uvel:\grad\varphi
-\rey\partial_1\uvel\cdot\varphi
+\np{\uvel\cdot\grad\uvel}\cdot\varphi}\,\dx\dt
=\int_\torus\int_\Omega f\cdot\varphi\,\dx\dt
\]
holds for all test functions $\varphi\in\CRcisigma(\torus\times\Omega)$.
\label{item:WeakSolution_NStp_WeakFormulation}
\end{enumerate}
\end{Definition}

The existence of weak solutions in the sense of Definition~\ref{def:WeakSolution_NStp}
satisfying an associated energy inequality
was shown in~\cite{Kyed_habil} for $\Omega=\R^3$.
Their asymptotic properties as $\snorm{x}\to\infty$
were investigated 
in~\cite{GaldiKyed_TPSolNS3D_AsymptoticProfile,
Eiter_SpatiallyAsymptoticStructureTPsolNSE_2021,
EiterGaldi_SpatialDecayVorticityTP_2021}.
For these results, it was necessary to ensure
higher regularity of the solution 
$\uvel$,
which was done by assuming that 
\begin{equation}
\label{el:serrin.proju}
\projcompl\uvel 
\in \LR{\rho}(\torus;\LR{\kappa}(\Omega)^3)
\end{equation}
holds for some $\kappa=\rho\in(5,\infty)$.
Moreover, it was shown 
in~\cite{Yang_EnergyEqualityWeakTPSolNSE_2022}
that $\uvel$ 
satisfies an energy equality
if
\eqref{el:serrin.proju} holds 
for some $\kappa\in[4,\infty]$ and $\rho\in[2,4]$ with $\frac{2}{\rho}+\frac{2}{\kappa}\leq 1$.
It is remarkable
that in both cases,
the additional integrability is only assumed for the purely periodic part $\projcompl\uvel$,
but not for the whole weak solution $\uvel$ 
as is done for the initial-value problem.
The main result of this article is in the same spirit
and can be seen as an extension of the regularity results used in~\cite{GaldiKyed_TPSolNS3D_AsymptoticProfile,
Eiter_SpatiallyAsymptoticStructureTPsolNSE_2021,
EiterGaldi_SpatialDecayVorticityTP_2021}.
More precisely, we consider the criteria
\begin{alignat}{2}
&\exists\, \kappa,\rho \in (1,\infty) \text{ with }
\frac{2}{\rho}+\frac{3}{\kappa}<1 :  &\quad
\projcompl\uvel&\in\LR{\rho}(\torus;\LR{\kappa}(\Omega)^3),
\label{el:reg.fct}
\\
&\exists\, \kappa,\rho \in (1,\infty) \text{ with }
\frac{2}{\rho}+\frac{3}{\kappa}<2 : &\quad
\grad\projcompl\uvel&\in\LR{\rho}(\torus;\LR{\kappa}(\Omega)^{3\times 3}).
\label{el:reg.grad}
\end{alignat}
If the domain has smooth boundary and the data are smooth,
then both lead to smooth solutions.

\begin{Theorem}\label{thm:smooth}
Let $\Omega\subset\R^3$ be an exterior domain
with boundary of class $\CR{\infty}$,
and let $\rey> 0$. 
Let
$f\in\CRci(\torus\times\Omega)$ and $\uvel_\ast\in\CRi(\torus\times\partial\Omega)$,
and let $\uvel$ be a
weak time-periodic solution to \eqref{sys:NavierStokesTP}
in the sense of Definition \ref{def:WeakSolution_NStp}
such that~\eqref{el:reg.fct} or~\eqref{el:reg.grad} is satisfied.
Then there exists a corresponding pressure field 
$\upres$
such that $(\uvel,\upres)$ is a smooth solution 
to~\eqref{sys:NavierStokesTP}
and
\[
\uvel\in\CRi(\torus\times\overline{\Omega})^3, 
\qquad 
\upres\in\CRi(\torus\times\overline{\Omega})
\]
\end{Theorem}

As an intermediate step, we show the following result
that also gives assumes less smooth data.

\begin{Theorem}\label{thm:regularity}
Let $\Omega\subset\R^3$ be an exterior domain
with boundary of class $\CR{2}$,
and let $\rey> 0$. 
Let $f$ and $\uvel_\ast$ be such that
\begin{subequations}
\label{cond:data}
\begin{align}
&\forall q,r\in(1,\infty): \ f\in\LR{r}(\torus;\LR{q}(\Omega)^3), 
\label{cond:f}
\\
&\uvel_\ast\in\CR{}\np{\torus;\CR{2}(\partial\Omega)^3}
\cap\CR{1}\np{\torus;\CR{}(\partial\Omega)^3}. 
\label{cond:BoundaryData}
\end{align}
\end{subequations}
Let $\uvel$ be a
weak time-periodic solution to \eqref{sys:NavierStokesTP}
in the sense of Definition \ref{def:WeakSolution_NStp}
such that 
\eqref{el:reg.fct} or \eqref{el:reg.grad} is satisfied.
Then $\vvel\coloneqq\proj\uvel$ and $\wvel\coloneqq\projcompl\uvel$ satisfy
\begin{align}
\forall s_2\in(1,\frac{3}{2}]:\, \quad
&\vvel\in\DSR{2}{s_2}(\Omega)^3,\ 
\label{el:reg.ss2}
\\
\forall s_1\in(\frac{4}{3},\infty]:\, \quad
&\vvel\in\DSR{1}{s_1}(\Omega)^3, \ 
\label{el:reg.ss1}
\\
\forall s_0\in(2,\infty]:\, \quad
&\vvel\in\LR{s_0}(\Omega)^3,
\label{el:reg.ss0}
\\
\forall q,r\in(1,\infty): \ \quad
&\wvel\in \WSR{1}{r}(\torus;\LR{q}(\Omega)^3)\cap\LR{r}(\torus;\WSR{2}{q}(\Omega)^3),
\label{el:reg.pp}
\end{align}
and there exists a pressure field $\upres\in\LRloc{1}(\torus\times\Omega)$ with 
$\vpres\coloneqq\proj\upres$ and $\wpres\coloneqq\projcompl\upres$
such that
\begin{equation}\label{el:reg.pres}
\forall s_2\in(1,\frac{3}{2}]:\
\vpres\in\DSR{2}{s_2}(\Omega)^3,
\qquad
\forall q,r\in(1,\infty): \ \grad\wpres\in\LR{r}(\torus;\LR{q}(\Omega)^3)
\end{equation}
and \eqref{sys:NavierStokesTP} 
is satisfied in the strong sense.

Additionally, 
if $\Omega$ has $\CR{3}$-boundary,
and if $\proj f\in\WSR{1}{q}(\Omega)^3$
and $\proj\uvel_\ast\in\WSR{3-1/q_1}{q_1}(\partial\Omega)^3$
for some $q_1\in(3,\infty)$, 
then 
\begin{equation}
\label{el:reg:ss2.full}
\forall s_2\in(1,\infty):\, \quad
\vvel\in\DSR{2}{s_2}(\Omega)^3,
\quad
\vpres\in\DSR{1}{s_2}(\Omega). 
\end{equation}
\end{Theorem}

Comparing the regularity criteria of Theorem~\ref{thm:smooth} and Theorem~\ref{thm:regularity}
with those used in~\cite{GaldiKyed_TPSolNS3D_AsymptoticProfile,
Eiter_SpatiallyAsymptoticStructureTPsolNSE_2021,
EiterGaldi_SpatialDecayVorticityTP_2021},
we see that the present article extends them
in two directions.
Firstly, by~\eqref{el:reg.fct} we extend the range of 
admissible parameters $\rho$, $\kappa$
in the sufficient condition~\eqref{el:serrin.proju}
by also allowing the mixed case $\rho\neq\kappa$.
Secondly,~\eqref{el:reg.grad}
is an alternative condition on certain integrability of 
the purely periodic part of the gradient 
$\grad\uvel$.
In particular, we can replace 
the assumption~\eqref{el:serrin.proju}
for some $\kappa=\rho\in(5,\infty)$
with one of the 
assumptions~\eqref{el:reg.fct}
or~\eqref{el:reg.grad}
in the main results of~\cite{GaldiKyed_TPSolNS3D_AsymptoticProfile,
Eiter_SpatiallyAsymptoticStructureTPsolNSE_2021,
EiterGaldi_SpatialDecayVorticityTP_2021},
and the results on the spatially asymptotic behavior
of the velocity and the vorticity field
derived there
are also valid under the alternative 
regularity criteria~\eqref{el:reg.fct} or \eqref{el:reg.grad}.

In Section~\ref{sec:Notation}
we next introduce the general notation used in this article.
In Section~\ref{sec:transference+embedding}
we recall the notion of Fourier multipliers in spaces with mixed Lebesgue norms
and introduce a corresponding transference principle,
from which we derive an embedding theorem.
Section~\ref{sec:prelimreg}
recalls a well-known regularity result for the steady-state Navier--Stokes equations, 
and it contains a similar result for the time-periodic Oseen problem,
which is a linearized version of \eqref{sys:NavierStokesTP}.
Finally, Theorem~\ref{thm:smooth} and Theorem~\ref{thm:regularity} 
will be proved in Section~\ref{sec:regularity}.

\section{Notation}
\label{sec:Notation}

For the whole article, the time period $\per>0$ is a fixed constant,
and $\torus\coloneqq\R/\per\Z$ denotes the corresponding torus group,
which severs as the time axis.
The spatial domain is usually given by a three-dimensional exterior domain $\Omega\subset\R^3$,
that is, the domain $\Omega$ is the complement of a compact connected set.
We write $\partial_t\uvel$ and $\partial_j\uvel\coloneqq\partial_{x_j}\uvel$
for partial derivatives with respect to time and space,
and we set $\Delta\uvel\coloneqq\partial_j\partial_j\uvel$ 
and $\Div\uvel\coloneqq\partial_j\uvel_j$,
where we used Einstein's summation convention.

We equip the compact abelian group $\torus$ with the normalized Lebesgue measure given by
\[
\forall f\in\CR{}(\torus):\qquad
\int_\torus f(t)\,\dt=\frac{1}{\per}\int_0^\per f(t)\,\dt,
\]
and the group $\Z$,
which can be identified with the dual group of $\torus$, with the counting measure. 
The Fourier transform $\FT_\grp$ on the locally compact group $\grp\colon\torus\times\Rn$, $n\in\N_0$,
and its inverse $\iFT_\grp$ are formally given by
\[
\begin{aligned}
\FT_\grp\nb{f}(k,\xi)
&\coloneqq\int_\torus\int_{\Rn} f(t,x)\e^{-i\perf kt-ix\cdot\xi}\,\dx\dt,\\
\iFT_\grp\nb{f}(t,x)
&\coloneqq\sum_{k\in\Z}\int_{\Rn} f(k,\xi)\e^{i\perf kt+ix\cdot\xi}\,\dxi,
\end{aligned}
\]
where the Lebesgue measure $\dxi$ is normalized appropriately
such that $\FT_\grp\colon\SR(\grp)\to\SR(\dualgrp)$ defines an isomorphism with inverse $\iFT_\grp$.
Here $\SR(\grp)$ is the so-called Schwartz--Bruhat space, which is a generalization 
of the classical Schwartz space in the Euclidean setting;
see \cite{Bruhat61,EiterKyed_tplinNS_PiFbook}.
By duality this induces an isomorphism $\FT_\grp\colon\TDR(\grp)\to\TDR(\dualgrp)$
of the dual spaces $\TDR(\grp)$ and $\TDR(\dualgrp)$,
the corresponding spaces of tempered distributions.

By $\LR{q}(\Omega)$
and $\WSR{m}{q}(\Omega)$ as well as
$\LR{q}(\torus\times\Omega)$
and $\WSR{m}{q}(\torus\times\Omega)$
we denote the classical Lebesgue and Sobolev spaces.
We define homogeneous Sobolev spaces by 
\[
\DSR{m}{q}(\Omega)
\coloneqq\setcl{\uvel\in\LRloc{1}(\Omega)}{\grad^m \uvel\in\LR{q}(\Omega)}, 
\] 
where $\grad^m\uvel$ denotes the collection of all (spatial) weak derivatives 
of $\uvel$ of $m$-th order.
We further set
\[
\CRcisigma(\Omega)\coloneqq\setc{\varphi\in\CRci(\Omega)^3}{\Div\varphi=0},
\]
where $\CRci(\Omega)$ is the class of 
compactly supported smooth functions on $\Omega$.
For $q\in[1,\infty]$ and a (semi-)normed vector space $X$,
$\LR{q}(\torus;X)$ denotes the corresponding Bochner-Lebesgue space on $\torus$,
and
\[
\WSR{1}{q}(\torus;X)
\coloneqq
\setcl{\uvel\in\LR{q}(\torus;X)}
{\partial_t\uvel\in\LR{q}(\torus;X)}.
\]
The projections
\[
\proj f\coloneqq \int_\torus f(t)\,\dt, 
\qquad \projcompl f\coloneqq f-\proj f
\] 
decompose $f\in\LR{1}\np{\torus;X}$ 
into a time-independent \emph{steady-state} part $\proj f$
and a \emph{purely periodic} part $\projcompl f$.

We further study the fractional time derivative 
$D_t^\alpha$ for $\alpha\in(0,\infty)$,
which is defined by
\[
D_t^\alpha\uvel(t)\coloneqq \iFT_\torus\bb{\snorm{\perft k}^\alpha\FT_\torus\nb{\uvel}}(t)
=\sum_{k\in\Z} \snorm{\perft k}^\alpha\uvel_k\e^{i\perf kt}
\]
for $\uvel\in\SR(\torus)$.
By Plancherel's theorem,
one readily verifies the integration-by-parts formula
\begin{equation}
\label{eq:fracder.ibp}
\int_\torus D_t^\alpha\uvel\, \vvel\,\dx
=\int_\torus\uvel\,D_t^\alpha\vvel\,\dx
\end{equation}
for all $\uvel,\vvel\in\SR(\torus)$.
By duality, $D_t^\alpha$ extends to an operator 
on the distributions $\TDR(\torus)$.
Note that in general we have $D_t^\alpha\uvel\neq\partial_t^\alpha\uvel$ for $\alpha\in\N$,
but it holds
\[
D_t^\alpha\uvel\in\LR{p}(\torus)
\iff 
\partial_t^\alpha\uvel\in\LR{p}(\torus)
\]
for $\alpha\in\N$ and $p\in(1,\infty)$.
If $\alpha=j/2$ for some $j\in\N$,
we usually write
$\sqrt{D}_t^j\uvel\coloneqq D_t^{j/2}\uvel$.

\section{Transference principle and embedding theorem}
\label{sec:transference+embedding}

To analyze mapping properties of the fractional derivative and other operators,
we need the notion of Fourier multipliers
on the locally compact abelian group
$\grp=\torus\times\Rn$ for $n\in\N_0$.
We are interested in multipliers that induce 
bounded operators between mixed-norm spaces
of the form $\LR{p}(\torus;\LR{q}(\Rn))$
for $p,q\in(1,\infty)$.
We call 
$M\in\LR{\infty}(\Z\times\Rn)$
an \textit{$\LR{p}(\torus;\LR{q}(\Rn))$-multiplier}
if there is $C>0$ such that
\[
\forall\uvel\in\SR(\torus\times\Rn):
\quad
\norml{\iFT_{\torus\times\Rn}\bb{M\,\FT_{\torus\times\Rn}\nb{\uvel}}}_{\LR{p}(\torus;\LR{q}(\Rn))}
\leq C \norm{\uvel}_{\LR{p}(\torus;\LR{q}(\Rn))},
\]
and we call 
$m\in\LR{\infty}(\R\times\Rn)$
an \textit{$\LR{p}(\R;\LR{q}(\Rn))$-multiplier}
if there is $C>0$ such that
\[
\forall\uvel\in\SR(\R\times\Rn):
\quad
\norml{\iFT_{\R\times\Rn}\bb{m\,\FT_{\R\times\Rn}\nb{\uvel}}}_{\LR{p}(\R;\LR{q}(\Rn))}
\leq C \norm{\uvel}_{\LR{p}(\R;\LR{q}(\Rn))}.
\]
The smallest such constant $C$ is denoted by 
$\norm{M}_{\calm_{p,q}(\torus\times\Rn)}$ 
and $\norm{m}_{\calm_{p,q}(\R\times\Rn)}$
and called the multiplier norm of $M$ and $m$,
respectively.
The following transference principle enables us to reduce multipliers on $\torus\times\Rn$ to 
multipliers on $\R\times\Rn$.

\begin{Proposition}
\label{prop:transference}
Let $p,q\in(1,\infty)$,
and let $m\in\CR{}(\torus\times\Rn)$ be an $\LR{p}(\R;\LR{q}(\Rn))$-multiplier.
Then $M\coloneqq m|_{\Z\times\Rn}$
is an $\LR{p}(\torus;\LR{q}(\Rn))$-multiplier
with norm
\[
\norm{M}_{\calm_{p,q}(\torus\times\Rn)}
\leq \norm{m}_{\calm_{p,q}(\R\times\Rn)}
\]
\end{Proposition}

\begin{proof}
The statement can be shown as in~\cite{Leeuw1965}, 
where a transference principle 
from scalar-valued $\LR{p}(\R)$-multipliers
to $\LR{p}(\torus)$-multipliers was shown.
For a more direct and modern approach, one may also follow the proof
of \cite[Proposition 5.7.1]{HytonenNeervenVeraarWeis_AnaBanachSpaces1_2016},
where an operator-valued version of
the result from~\cite{Leeuw1965} was established.
\end{proof}

We now apply this transference principle
to show the following result, which is 
an extension of \cite[Theorem 4.1]{GaldiKyed_TPflowPastBody_2018}
to the case of mixed norms
Moreover, we also take fractional time derivatives into account.

\begin{Theorem}\label{thm:embedding}
Let $\Omega\subset\Rn$, $n\geq 2$, be a bounded or exterior domain with Lipschitz boundary,
and let $q,r\in(1,\infty)$.
For $\alpha\in[0,2]$ let
\[
r_0\in
\begin{cases}
\bb{1,\frac{2r}{2-\alpha r}} 
&\tif \alpha r <2, 
\\
[1,\infty) 
&\tif \alpha r =2,
\\
[1,\infty] 
&\tif \alpha r >2,
\end{cases}
\qquad\qquad
q_0
\in
\begin{cases}
\bb{q,\frac{nq}{n-(2-\alpha)q}} 
&\tif (2-\alpha) q <n, 
\\
[q,\infty) 
&\tif (2-\alpha) q =n,
\\
[q,\infty] 
&\tif (2-\alpha) q >n,
\end{cases}
\]
and for $\beta\in[0,1]$ let
\[
r_1\in
\begin{cases}
\bb{1,\frac{2r}{2-\beta r}} 
&\tif \beta r <2, 
\\
[1,\infty) 
&\tif \beta r =2,
\\
[1,\infty] 
&\tif \beta r >2,
\end{cases}
\qquad\qquad
q_1
\in
\begin{cases}
\bb{q,\frac{nq}{n-(1-\beta)q}} 
&\tif (1-\beta) q <n, 
\\
[q,\infty) 
&\tif (1-\beta) q =n,
\\
[q,\infty] 
&\tif (1-\beta) q >n.
\end{cases}
\]
Then there $C=C(n,q,r,\alpha,\beta)>0$
such that all 
$\uvel\in\WSR{1}{r}(\torus;\LR{q}(\Omega))\cap\LR{r}(\torus;\WSR{2}{q}(\Omega))$ 
satisfy the inequality
\begin{equation}\label{est:embedding}
\begin{aligned}
\norm{\uvel}_{\LR{r_0}(\torus;\LR{q_0}(\Omega))}
+\norm{\grad\uvel}_{\LR{r_1}(\torus;\LR{q_1}(\Omega))}
&+\norm{\sqrt{D}_t\uvel}_{\LR{r_1}(\torus;\LR{q_1}(\Omega))}
+\norm{\sqrt{D}_t\grad\uvel}_{\LR{r}(\torus;\LR{q}(\Omega))}
\\
&\qquad
\leq C\bp{
\norm{\uvel}_{\WSR{1}{r}(\torus;\LR{q}(\Omega))}
+\norm{\uvel}_{\LR{r}(\torus;\WSR{2}{q}(\Omega))}}.
\end{aligned}
\end{equation}
\end{Theorem}

\begin{proof}
For the proof
we proceed analogously to \cite[Theorem 4.1]{GaldiKyed_TPflowPastBody_2018}.
However, we have to modify some arguments in the case $p\neq q$, 
and we also derive estimates for the fractional time derivative,
which is why we give some details here.
Using Sobolev extension operators 
and the density properties of $\SR(\torus\times\Rn)$, it suffices to show estimate \eqref{est:embedding}
for $\Omega=\Rn$ and $\uvel\in\SR(\grp)$ with $\grp=\torus\times\Rn$.

We begin with the estimate of $\uvel$.
By means of the Fourier transform,
we obtain
\begin{equation}\label{eq:embeddingthm.decomp}
\begin{split}
\uvel
=\iFT_\grp\Bb{
\frac{1}{1+\snorm{\xi}^2+i \perf k} \FT_G\bb{\uvel+\partial_t\uvel-\Delta\uvel}
} 
=
\bp{\gamma_{\alpha/2}\otimes\Gamma_{2-\alpha}}
\ast
F,
\end{split}
\end{equation}
where 
\[
\begin{aligned}
\gamma_\mu
&\coloneqq
\iFT_{\torus}\bb{
\bp{1-\delta_\Z(k)}
\snorml{\perft k}^{-\mu}
},
&\qquad
\Gamma_\nu
&\coloneqq
\iFT_{\Rn}\bb{
\np{1+\snorm{\xi}^2}^{-\nu/2}
},
\\
F&\coloneqq
\iFT_\grp \bb{M\,\FT_G\nb{\uvel+\partial_t\uvel-\Delta\uvel}},
&\qquad
M(k,\xi)
&\coloneqq
\frac{\np{1+\snorm{\xi}^2}^{1-\alpha/2}
{\snorml{\perf k}^{\alpha/2}}}
{1+\snorm{\xi}^2+i\perf k}.
\end{aligned}
\]
Here $\delta_\Z$ is the delta distribution on $\Z$,
that is, $\delta_\Z\colon\Z\to\set{0,1}$ 
with $\delta_\Z(k)=1$ if and only if
$k=0$.
We can extend $M\colon\Z\times\Rn\to\C$ to a continuous function $m\colon\R\times\Rn\to\C$
in a trivial way such that $M=m|_{\Z\times\Rn}$.
One readily shows that $m$ satisfies the 
Lizorkin multiplier theorem~\cite[Corollary 1]{Lizorkin_MultipliersMixedNorms1970},
the function $m$ is an $\LR{r}(\R;\LR{q}(\Rn))$-multiplier.
Due to the transference principle from Proposition~\ref{prop:transference},
this implies
that $M$ is an $\LR{r}(\torus;\LR{q}(\Rn))$-multiplier,
and we have
\begin{equation}
\label{eq:embeddingthm.mult}
\begin{aligned}
\norm{F}_{\LR{r}(\torus;\LR{q}(\Rn))}
&\leq C\norm{\uvel+\partial_t\uvel-\Delta\uvel}_{\LR{r}(\torus;\LR{q}(\Rn))}
\\
&\leq C\bp{
\norm{\uvel}_{\WSR{1}{r}(\torus;\LR{q}(\Rn))}
+\norm{\uvel}_{\LR{r}(\torus;\WSR{2}{q}(\Rn))}}.
\end{aligned}
\end{equation}
Moreover, 
from~\cite[Example 3.1.19]{Grafakos1}
and~\cite[Proposition 6.1.5]{Grafakos2}
we conclude
\[
\begin{aligned}
\gamma_\mu&\in\LR{\frac{1}{1-\mu},\infty}(\torus),
&\qquad
\forall s\in\big[1,\frac{1}{1-\mu}\big):\ \gamma_\mu&\in\LR{s}(\torus),
\\
\Gamma_\nu&\in\LR{\frac{n}{n-\nu},\infty}(\Rn),
&\qquad
\forall s\in\big[1,\frac{n}{n-\nu}\big): \ \Gamma_\nu&\in\LR{s}(\Rn)
\end{aligned}
\]
for $\mu\in(0,1)$ and $\nu\in(0,n)$.
Young's inequality thus implies that 
$\varphi\mapsto\gamma_{\alpha/2}\ast\varphi$
defines a continuous linear operator $\LR{r}(\torus)\to\LR{r_0}(\torus)$ if $r_0\geq r$,
and
$\psi\mapsto\Gamma_{2-\alpha}\ast\psi$
defines a continuous linear operator $\LR{q}(\torus)\to\LR{q_0}(\torus)$.
Therefore, 
formula~\eqref{eq:embeddingthm.decomp}
yields
\[
\begin{aligned}
\norm{\uvel}_{\LR{r_0}(\torus;\LR{q_0}(\Rn))}
&=\Bp{\int_{\torus}\normL{
\int_\torus
\gamma_\alpha(t-s)
\Gamma_{2-\alpha}
\ast_{\Rn}
F(s,\cdot)\,\ds}_{q_0}^{r_0}\,\dt}^{\frac{1}{r_0}} 
\\
&\leq\Bp{
\int_{\torus} \Bp{
\int_\torus\snorml{\gamma_\alpha(t-s)}\,
\norml{\Gamma_{2-\alpha}
\ast_{\Rn}
F(s,\cdot)}_{q_0}\,\ds}^{r_0}\,\dt}^{\frac{1}{r_0}} 
\\
&\leq C\Bp{
\int_{\torus} 
\norml{\Gamma_{2-\alpha}
\ast_{\Rn}
F(t,\cdot)}_{q_0}^{r}\,\dt}^{\frac{1}{r}} \\
&\leq C\norm{F}_{\LR{r}(\torus;\LR{q}(\Rn))}.
\end{aligned}
\]
Invoking now~\eqref{eq:embeddingthm.mult},
we arrive at the desired estimate for $\uvel$ if $r_0\geq r$. 
Since $\torus$ is compact, the estimate for $r_0<r$ follows immediately.

The remaining estimates of $\grad\uvel$,
$\sqrt{D}_t\uvel$ and $\sqrt{D}_t\grad\uvel$
can be shown in the same way as those for $\uvel$.
Note that for the estimates of $\sqrt{D}_t\uvel$ and $\sqrt{D}_t\grad\uvel$,
the procedure has to be slightly modified
since the trivial extension of the
corresponding multipliers to $\R\times\R^n$ 
is not continuous.
To demonstrate this, we focus 
on the estimate for $\sqrt{D}_t\grad\uvel$,
which
means nothing else than the boundedness
of the linear operator
\[
\sqrt{D}_t\grad\colon
\WSR{1}{r}(\torus;\LR{q}(\Rn))\cap
\LR{r}(\torus;\WSR{2}{q}(\Rn))
\to\LR{r}(\torus;\LR{q}(\Rn)).
\]
Similarly to above, 
this boundedness follows if the function
\[
M\colon\Z\times\R^n\to\C,\qquad
M(k,\xi)=\frac{\snorm{\perf k}^\half \xi_j}{\snorm{\xi}^2+i\perf k}
\]
is an $\LR{r}(\torus;\LR{q}(\Rn))$-multiplier
for $j=1,\dots,n$.
Note that its trivial extension is not a continuous function in $(0,0)\in\R\times\Rn$,
which is necessary for application of the transference principle 
from Proposition~\ref{prop:transference}.
However, since $M(0,\xi)=0$, we can introduce
a smooth cut-off function $\chi\in\CRci(\R)$ 
with $\supp\chi\subset(-1,1)$ and such that
$\chi(\eta)=1$ for $\snorm{\eta}\leq \frac{1}{2}$.
We define
\[
m\colon\R\times\R^n\to\C,\qquad
m(\eta,\xi)=\frac{\bp{1-\chi(\eta)}\snorm{\perf\eta}^\half \xi_j}{\snorm{\xi}^2+i\perf\eta}.
\]
Then $m$ is a smooth function with $m|_{\Z\times\Rn}=M$,
and one readily verifies that $m$
satisfies the multiplier theorem 
by Lizorkin~\cite[Corollary 1]{Lizorkin_MultipliersMixedNorms1970}.
Finally, Proposition~\ref{prop:transference}
shows that 
$M$ is an $\LR{r}(\torus;\LR{q}(\Rn))$-multiplier,
which implies the estimate for 
$\sqrt{D}_t\grad\uvel$.
\end{proof}

As mentioned in the proof, the lower bound $1$ for $r_0$ and $r_1$ is valid
since the torus $\torus$ has finite measure.
In the same manner, the lower bound for $q_0$ and $q_1$ can be replaced with $1$ 
if $\Omega$ is a bounded domain.

In \cite[Theorem 4.1]{EiterKyed_ViscousFlowAroundRigidBodyPerformingTPMotion_2021}
a homogeneous version of Theorem \ref{thm:embedding} was shown, but 
only in the case $q=r$. 
Modifying the proof 
in~\cite{EiterKyed_ViscousFlowAroundRigidBodyPerformingTPMotion_2021}
and using similar arguments as above, 
this result is easily extended to the case $q\neq r$.

We might also formulate the assumptions 
on the integrability exponents in Theorem \ref{thm:embedding} as follows:
Let $r_0,p_0,r_1,p_1\in[1,\infty]$ such that
\[
\begin{aligned}
\frac{2}{r}-\alpha &< \frac{2}{r_0} \leq 2,
&\qquad\qquad
\frac{n}{q}-(2-\alpha) &< \frac{n}{q_0} \leq \frac{n}{q},
\\
\frac{2}{r}-\beta &< \frac{2}{r_1} \leq 2,
&\qquad\qquad
\frac{n}{q}-(1-\beta) &< \frac{n}{q_1} \leq \frac{n}{q},
\end{aligned}
\]
where in each of the four conditions the left $<$ can be replaced with $\leq$
if the respective lower bound is different from $0$.

\section{Preliminary regularity results}
\label{sec:prelimreg}

As a preparation for the proof of the main theorems,
we first consider
the steady-state Navier--Stokes equations
\begin{subequations}\label{sys:statNS}
\begin{alignat}{2}
-\Delta\vvel-\rey\partial_1\vvel+\vvel\cdot\grad\vvel+\grad\vpres
&=F
&&\tin\Omega, \\
\Div\vvel&=0 
&&\tin\Omega, \\
\vvel&=\vvel_\ast 
&&\ton\partial\Omega
\end{alignat}
\end{subequations}
and recall 
the following result on the regularity of weak solutions.

\begin{Theorem}\label{thm:statNS}
Let $\Omega\subset\R^3$ be an exterior domain with $\CR{2}$-boundary.
Let $q_0\in(1,2)$ such that
\begin{equation}\label{el:statdata}
F\in\LR{q}(\Omega)^3, \qquad
\vvel_\ast\in\WSR{2-\frac{1}{q}}{q}(\partial\Omega)^3
\end{equation}
for $q=q_0$ and for $q=\frac{3}{2}$.
If $\vvel$ is a weak solution to \eqref{sys:NavierStokesTP_sspart},
then there exists an associated pressure field $\vpres$ such that
\[
\vvel\in\DSR{2}{q_0}(\Omega)^3\cap
\DSR{1}{4q_0/(4-q_0)}(\Omega)^3\cap
\LR{2q_0/(2-q_0)}(\Omega)^3,
\qquad \vpres\in\DSR{1}{q_0}(\Omega),
\]
and \eqref{sys:statNS} is satisfied in the strong sense.
Additionally, 
if there exists $q_1\in(3,\infty)$ such that \eqref{el:statdata} holds
for all 
$q\in(1,q_1]$,
then $\vvel$ satisfies
\eqref{el:reg.ss1} and \eqref{el:reg.ss0}.
Moreover, if $\Omega$ has $\CR{3}$-boundary and $F\in\WSR{1}{q_1}(\Omega)^3$ and 
$\vvel_\ast\in\WSR{3-1/q_1}{q_1}(\partial\Omega)^3$,
then \eqref{el:reg.ss2} holds and
$\vpres\in\DSR{1}{q}(\Omega)$ for all $q\in(1,\infty]$.
\end{Theorem}

\begin{proof}
See \cite[Lemma X.6.1 and Theorem X.6.4]{GaldiBookNew}.
\end{proof}

We further derive a similar regularity result for
weak solutions to the time-periodic Oseen problem, 
which is the linearization 
of~\eqref{sys:NavierStokesTP}
given by
\begin{subequations}\label{sys:Oseen}
\begin{alignat}{2}
\partial_t\uvel-\Delta\uvel - \rey\partial_1\uvel +\grad\upres &= f 
&&\qquad \tin\torus\times\Omega, 
\\
\Div\uvel &=0 
&&\qquad \tin\torus\times\Omega,
\\
\uvel & =\uvel_\ast
&&\qquad\ton\torus\times\partial\Omega.
\end{alignat}
\end{subequations}
Here we focus on the case of purely oscillatory data.
To shorten the notation, we denote the mixed-norm parabolic space by
\[
\calw_{q,r}\coloneqq
\WSR{1}{r}(\torus;\LR{q}(\Omega)^3)\cap\LR{r}(\torus;\WSR{2}{q}(\Omega)^3).
\]

\begin{Lemma}\label{lem:OseenReg}
Let $\Omega\subset\R^3$ be an exterior domain of class $\CR{2}$,
let $\uvel_\ast$ be as 
in~\eqref{cond:BoundaryData},
and let $f\in\LR{r}(\torus;\LR{q}(\Omega)^3)$ for some $r,q\in(1,\infty)$
such that $\proj f=0$ and $\proj \uvel_\ast=0$.
Let
$\uvel\in\LR{\infty}(\torus;\LR{2}(\Omega)^3)$ with
$\grad\uvel\in\LR{2}(\torus\times\Omega)^{3\times3}$
and $\proj\uvel=0$
be a weak 
solution to \eqref{sys:Oseen},
that is, $\uvel=\uvel_\ast$ on $\torus\times\partial\Omega$,
$\Div\uvel=0$ and 
\begin{equation}\label{eq:NSlintp_weak}
\int_{\torus}\int_\Omega\bb{-\uvel\cdot\partial_t\varphi
+\grad\uvel:\grad\varphi
-\rey\partial_1\uvel\cdot\varphi
}\,\dx\dt
=\int_{\torus}\int_\Omega f\cdot\varphi\,\dx\dt
\end{equation}
for all $\varphi\in\CRcisigma(\torus\times\Omega)$.
Then 
$\uvel\in\calw_{q,r}$,
and there exists $\upres\in\LR{q}(\torus;\DSR{1}{q}(\Omega))$
such that $\np{\uvel,\upres}$ is a strong solution to~\eqref{sys:Oseen}.
\end{Lemma}

\begin{proof}
For $q=r$, the result was shown
in~\cite[Lemma 5.1]{EiterGaldi_SpatialDecayVorticityTP_2021}.
Arguing in the same way, 
we can show that it suffices to treat the case $\uvel_\ast=0$.
In this case, first consider 
a solution of the time-periodic Stokes problem,
that is, the system
\begin{subequations}
\label{sys:Stokes}
\begin{alignat}{2}
\partial_t\Uvel-\Delta\Uvel +\grad\Upres &= f 
&&\qquad \tin\torus\times\Omega, 
\\
\Div\Uvel &=0 
&&\qquad \tin\torus\times\Omega,
\\
\Uvel & =0
&&\qquad\ton\torus\times\partial\Omega.
\end{alignat}
\end{subequations}
We now use the result from~\cite[Theorem~5.5]{EiterKyedShibata_PeriodicLpEstRboundedness}
on maximal regularity for this system
for right-hand sides in 
$\LR{r}(\torus;\LR{q}(\Omega)^3)$.
From this,
we conclude the existence of a unique solution
$(\Uvel,\Upres)$
with $\proj\Uvel=0$ and 
$\Uvel\in\calw_{q,r}$.
The embedding theorem~\ref{thm:embedding}
implies
that
$\partial_1\Uvel\in\LR{\tilde r}(\torus;\LR{q}(\Omega))$
for $\tilde r\in(1,\infty)$ with $\frac{1}{\tilde r}\in(\frac{1}{r}-\frac{1}{2},1]$.
We again employ the the maximal regularity result 
from~\cite[Theorem~5.5]{EiterKyedShibata_PeriodicLpEstRboundedness}
to obtain the existence of 
a unique solution  $\np{\Vvel,\Vpres}$ to
\[
\begin{aligned}
\partial_t\Vvel-\Delta\Vvel +\grad\Vpres &= \tau\partial_1\Uvel
&&\tin\torus\times\Omega, 
\\
\Div\Vvel &=0 
&&\tin\torus\times\Omega,
\\
\Vvel & =0
&&\ton\torus\times\partial\Omega,
\end{aligned}
\]
such that $\proj\Vvel=0$ and 
$\Vvel\in\calw_{q,\tilde r}$
for all $\tilde r$ as above.
Employing Theorem~\ref{thm:embedding} once more,
we see that $\partial_1\Vvel\in\LR{\hat r}(\torus;\LR{q}(\Omega))$
for any $\hat{r}\in(1,\infty)$.
In particular, we can choose $\hat r=q$,
that is, we have
$\partial_1\Vvel\in\LR{q}(\torus\times\Omega)$.
Now we can use the maximal regularity result
\cite[Theorem 5.1]{GaldiKyed_TPflowViscLiquidpBody_2018}
for the Oseen system for right-hand sides in 
$\LR{q}(\torus\times\Omega)$
to find a solution $(\Wvel,\Wpres)$ 
to
\[
\begin{aligned}
\partial_t\Wvel-\Delta\Wvel - \rey\partial_1\Wvel +\grad\Wpres &= \tau\partial_1\Vvel
&&\tin\torus\times\Omega, 
\\
\Div\Wvel &=0 
&&\tin\torus\times\Omega,
\\
\Wvel & =0
&&\ton\torus\times\partial\Omega,
\end{aligned}
\]
such that $\Wvel\in\calw_{q,q}$.
Theorem~\ref{thm:embedding} further implies
$\partial_1\Wvel\in\LR{\overline r}(\torus;\LR{q}(\Omega))$
for $\frac{1}{\overline r}\in(\frac{1}{q}-\frac{1}{2},1]$
such that
$\Wvel\in\calw_{\overline r,q}$
by~\cite[Theorem~5.5]{EiterKyedShibata_PeriodicLpEstRboundedness}.
Repeating this argument once again,
we obtain $\Wvel\in\calw_{q,\overline r}$
for $\overline r\in(1,\infty)$.
In total, we see that 
$\tuvel\coloneqq\Uvel+\Vvel+\Wvel$
and $\tupres\coloneqq\Upres+\Vpres+\Wpres$
satisfy the Oseen system \eqref{sys:Oseen}
and $\tuvel\in\calw_{q,r}$. 
To conclude that $\uvel=\tuvel$
one can now proceed as in the proof of
in~\cite[Lemma 5.1]{EiterGaldi_SpatialDecayVorticityTP_2021}.
The regularity of the pressure $\upres$ follows immediately.
\end{proof}

Observe that for the proof we combined two results on maximal regularity: 
one for the Stokes problem \eqref{sys:Stokes}
for right-hand sides in $\LR{r}(\torus;\LR{q}(\Omega)^3)$,
and one for the Oseen problem~\eqref{sys:Oseen}
for right-hand sides in $\LR{q}(\torus\times\Omega)^3$.
The argument could be shortened severely
if such a result would be available 
for the Oseen problem~\eqref{sys:Oseen}
for right-hand sides in $\LR{r}(\torus;\LR{q}(\Omega)^3)$.
For a proof, one can use the approach
developed in~\cite[Theorem~5.5]{EiterKyedShibata_PeriodicLpEstRboundedness},
which would also give corresponding \textit{a priori} estimates.

\section{Regularity of time-periodic weak solutions}
\label{sec:regularity}

Now we begin with the proof of
Theorem~\ref{thm:regularity}, 
for which we proceed by a bootstrap argument
that increases the range of admissible integrability exponents step by step.
To shorten the notation,
we introduce the \textit{Serrin number}
\[
s_{q,r}\coloneqq\frac{2}{r}+\frac{3}{q}.
\]
For the whole section, let 
$f$ and $\uvel_\ast$ satisfy~\eqref{cond:data},
and let $\uvel$ be a 
weak solution in the sense of Definition~\ref{def:WeakSolution_NStp}.
We decompose $\uvel$ and set $\vvel\coloneqq\proj\uvel$ and $\wvel\coloneqq\projcompl\uvel$.

We first show 
that the definition of weak solutions already implies some degree 
of increased regularity
and that there exists a pressure such that
the Navier--Stokes equations are satisfied in the strong sense.

\begin{Lemma}\label{lem:reg.initial}
There exists a pressure field
$\upres=\vpres+\wpres$ such that
\begin{align}
\label{el:regularity_ss1}
&\forall s_2\in(1,\frac{3}{2}]:\, \vvel\in\DSR{2}{s_2}(\Omega)^3,\ \ 
\vpres\in\DSR{1}{s_2}(\Omega)^3,
\\
\label{el:regularity_ss2}
&\forall s_1\in(\frac{4}{3},3]:\, \vvel\in\DSR{1}{s_1}(\Omega)^3, \qquad
\forall s_0\in(2,\infty):\, \vvel\in\LR{s_0}(\Omega)^3,
\\
\label{el:regularity_pp1}
&\forall r,q\in(1,\infty) \text{ with }
s_{q,r}=4:
\quad
\wvel\in\calw_{q,r}, \quad 
\wpres\in\LR{r}(\torus;\DSR{1}{q}(\Omega)),
\end{align}
and the Navier--Stokes equations \eqref{sys:NavierStokesTP} are satisfied in the strong sense.
More precisely, it holds
\begin{subequations}\label{sys:NavierStokesTP_sspart}
\begin{alignat}{2}
-\Delta\vvel-\rey\partial_1\vvel+\vvel\cdot\grad\vvel+\grad\vpres
&=\proj f-\proj\nb{\wvel\cdot\grad\wvel}
&&\qquad\tin\Omega, \\
\Div\vvel&=0 
&&\qquad\tin\Omega, \\
\vvel&=\proj\uvel_\ast 
&&\qquad\ton\partial\Omega.
\end{alignat}
\end{subequations}
and
\begin{subequations}\label{sys:NavierStokesTP_pppart}
\begin{alignat}{2}
\partial_t\wvel-\Delta\wvel - \rey\partial_1\wvel +\grad\wpres 
&= \projcompl f -\vvel\cdot\grad\wvel-\wvel\cdot\grad\vvel
-\projcompl\np{\wvel\cdot\grad\wvel}
&& \tin\torus\times\Omega, 
\\
\Div\wvel &=0 && \tin\torus\times\Omega,
\\
\wvel & =\projcompl\uvel_\ast
&&\ton\torus\times\partial\Omega.
\end{alignat}
\end{subequations}
\end{Lemma}

\begin{proof}
From the integrability of $\wvel$ 
we conclude by H\"older's inequality that 
$\wvel\cdot\grad\wvel
\in\LR{1}(\torus;\LR{3/2}(\Omega))\cap\LR{2}(\torus;\LR{1}(\Omega))$.
We thus have
$\proj f-\proj(\wvel\cdot\grad\wvel)\in\LR{1}(\Omega)\cap\LR{3/2}(\Omega)$,
and Theorem~\ref{thm:statNS}
yields the existence of $p$ such that \eqref{sys:NavierStokesTP_sspart}
as well as \eqref{el:regularity_ss1} and \eqref{el:regularity_ss2} hold.

To obtain the regularity statement for $\wvel$,
note that 
\eqref{el:regularity_ss2} implies
$\vvel\cdot\grad\wvel\in\LR{2}(\torus;\LR{q}(\Omega)^3)$
for all $q\in(1,2)$. 
Moreover, 
we have 
$\wvel\in\LR{2}(\torus;\LR{6}(\Omega)^3)\cap\LR{\infty}(\torus;\LR{2}(\Omega)^3)
\hookrightarrow \LR{r}(\torus;\LR{q}(\Omega)^3)$
for all $r\in[2,\infty]$ and $q\in[2,6]$
with $s_{q,r}=\frac{3}{2}$
by the Sobolev inequality and interpolation.
In virtue of \eqref{el:regularity_ss2} and H\"older's inequality,
we conclude
$\wvel\cdot\grad(\vvel+\wvel)\in \LR{r}(\torus;\LR{q}(\Omega)^3)$
for all $q\in(1,\frac{3}{2}]$ and $r\in[1,2)$ with $s_{q,r}=4$.
In consequence, we obtain
\[
\projcompl f -\vvel\cdot\grad\wvel-\wvel\cdot\grad\vvel
-\projcompl\np{\wvel\cdot\grad\wvel} \in
\LR{r}(\torus;\LR{q}(\Omega)^3)
\]
for all such $q$ and $r$.
Now
Lemma~\ref{lem:OseenReg}
yields the existence of a pressure $\wpres$ such that \eqref{sys:NavierStokesTP_pppart}
is satisfied in the strong sense
and
\eqref{el:regularity_pp1} holds.
\end{proof}

In the following lemmas, 
we always assume that $\wvel\in\calw_{q, r}$ for some given $q,r\in(1,\infty)$,
and the goal is to extend the range of one of the parameters $q$ or $r$ 
while the other one remains fixed. 
We use the assumption on 
additional regularity \eqref{el:reg.fct} or \eqref{el:reg.grad},
or the embedding properties from Theorem \ref{thm:embedding}
to conclude
\begin{equation}\label{el:w}
\wvel\in\LR{r_0}(\torus;\LR{q_0}(\Omega)^3)
\end{equation}
for a class of parameters $q_0,r_0\in[1,\infty]$,
and 
\begin{equation}\label{el:gradw}
\grad\wvel\in\LR{r_1}(\torus;\LR{q_1}(\Omega)^{3\times 3})
\end{equation}
for a class of parameters $q_1,r_1\in[1,\infty]$.
and we use Lemma \ref{lem:reg.initial} or Theorem~\ref{thm:statNS}
to deduce
\begin{equation}\label{el:v}
\vvel\in\LR{s_0}(\Omega)^3
\end{equation}
for certain $s_0\in[1,\infty]$
and 
\begin{equation}\label{el:gradv}
\grad\vvel\in\LR{s_1}(\Omega)^{3\times 3}
\end{equation}
for certain $s_1\in[1,\infty]$.
Then H\"older's inequality yields
suitable estimates of the nonlinear terms
and of the total right-hand side
\begin{equation}\label{el:rhs.tp}
\projcompl f -\vvel\cdot\grad\wvel-\wvel\cdot\grad\vvel
-\projcompl\np{\wvel\cdot\grad\wvel} \in
\LR{r_5}(\torus;\LR{q_5}(\Omega)^3)
\end{equation}
for a certain class of parameters $q_5,r_5\in(1,\infty)$.
Invoking now the regularity result from
Lemma~\ref{lem:OseenReg},
we conclude
$\wvel\in\calw_{q_5,r_5}$.

As a preparation, we first derive suitable estimates of the nonlinear terms
if we have $\wvel\in\calw_{q,r}$. 
In the next lemma we start with the nonlinear term 
\begin{equation}\label{el:wgradw}
\wvel\cdot\grad\wvel\in\LR{r_2}(\torus;\LR{q_2}(\Omega)^3)
\end{equation}
and we show better integrability 
for $\vvel$ and $\grad\vvel$ for sufficiently large $q$.

\begin{Lemma}\label{lem:wgradw}
Let $\wvel\in\calw_{q,r}$ 
for some $q,r\in(1,\infty)$. 
Then
\eqref{el:wgradw} for
\begin{enumerate}
\item[i.]
$\frac{3}{q_2}\in\big(\max\setl{0,s_{q,r}-1,\frac{3}{q}+s_{q,r}-2},\min\setl{3,\frac{6}{q}}\big]$
and $r_2=r$, and
\item[ii.]
$q_2=q$ and
$\frac{2}{r_2}\in\big(\max\setl{0,s_{q,r}-1},2\big]$.
\end{enumerate}
Moreover, 
if $s_{q,r}<\frac{3}{2}+\frac{1}{\max\set{2,r}}$ or $q>3$, then
the steady-state part
$\vvel$ satisfies \eqref{el:reg.ss1} and \eqref{el:reg.ss0}.
\end{Lemma}

\begin{proof}
At first, 
Theorem~\ref{thm:embedding} yields
\eqref{el:w} for $r_0=\infty$ and 
$\frac{3}{q_0}\in\big(\max\setl{0,s_{q,r}-1},\frac{3}{q}\big]$,
and \eqref{el:gradw} for $r_1=r$ and
$\frac{3}{q_1}\in\big(\max\setl{0,\frac{3}{q}-1},\frac{3}{q}\big]$,
so that we deduce 
\eqref{el:wgradw} for $r_2=r$ and 
$q_2$ as asserted in i.
Moreover, Theorem \ref{thm:embedding}
yields
$\wvel\in\LR{r_0}(\torus;\LR{q}(\Omega)^3)$ for $r_0\in[1,\infty]$ 
as well as 
$\grad\wvel\in\LR{r_1}(\torus;\LR{\infty}(\Omega)^{3\times 3})$
for
$\frac{2}{r_1}
\in\big(\max\setl{0,s_{q,r}-1},2\big]$.
Now H\"older's inequality implies the
integrability of $\wvel\cdot\grad\wvel$ asserted in ii.

If additionally $s_{q,r}<\frac{3}{2}+\frac{1}{\max\set{2,r}}$, then 
the lower bound in i.~is smaller than $1$,
so that 
$\proj f-\proj(\wvel\cdot\grad\wvel)\in\LR{q_2}(\Omega)^3$
for some $q_2\in(3,\infty)$.
The same follows from ii.~for $q_2=q>3$.
Now Theorem \ref{thm:statNS}
yields \eqref{el:reg.ss1} and \eqref{el:reg.ss0}.
\end{proof}

Next we treat the nonlinear terms that involves $\vvel$ and $\grad\vvel$, namely
we show that $\wvel\in\calw_{q,r}$ implies
\begin{equation}\label{el:vgradw}
\vvel\cdot\grad\wvel\in\LR{r_3}(\torus;\LR{q_3}(\Omega)^3)
\end{equation}
and
\begin{equation}\label{el:wgradv}
\wvel\cdot\grad\vvel\in\LR{r_4}(\torus;\LR{q_4}(\Omega)^3)
\end{equation}
for suitable parameters $q_3,r_3,q_4,r_4\in[1,\infty]$.

\begin{Lemma}\label{lem:vgradwwgradv}
Let $\wvel\in\calw_{q,r}$ 
for some $q,r\in(1,\infty)$. 
Then
\eqref{el:vgradw} holds for
\begin{enumerate}
\item[i.]
$\frac{3}{q_3}\in \bp{\max\setl{0,\frac{3}{q}-1},\min\setl{3,\frac{3}{q}+\frac{3}{2}}}$
and $r_3=r$, and
\item[ii.]
$q_3=q$ and
$\frac{2}{r_3}\in\big(\max\setl{0,\frac{2}{r}-1},2\big]$,
\end{enumerate}
and 
\eqref{el:wgradv} holds for 
\begin{enumerate}
\item[iii.]
$\frac{3}{q_4}\in 
\bp{\max\set{\frac{3}{4},\frac{3}{q}-\frac{3}{4},\frac{6}{q}-\frac{11}{4}},\min\set{3,\frac{3}{q}+\frac{9}{4}}}$
and $r_4=r$, and
\item[iv.]
$\frac{3}{q_4}\in 
\bp{0,\min\set{3,\frac{3}{q}+\frac{9}{4}}}$
and $r_4=r$ if $q> 3$, and
\item[v.]
$q_4=q$ and
$\frac{2}{r_4}\in\big(\max\set{0,\frac{2}{r}-1},2\big]$.
\end{enumerate}
\end{Lemma}

\begin{proof}
Theorem~\ref{thm:embedding} implies
\eqref{el:gradw} for $\frac{3}{q_1}\in\big(\max\set{0,\frac{3}{q}-1},\frac{3}{q}\big]$ and $r_1=r$
as well as for $\frac{3}{q_1}=\frac{3}{q}-\delta$ 
and $\frac{2}{r_1}\in\big(\max\set{0,\frac{2}{r}-(1-\delta)},\frac{2}{r}\big]$
for $\delta>0$ small.
Moreover, by Lemma \ref{lem:reg.initial} we have \eqref{el:v} for $s_0\in(2,\infty)$,
and H\"older's inequality implies
\eqref{el:vgradw} for $q_3$ and $r_3$ as in i.~or ii.

Theorem~\ref{thm:statNS} and Lemma~\ref{lem:reg.initial} yield
and \eqref{el:gradv} for all 
$\frac{1}{s_1}\in\bp{\max\set{\frac{1}{4},\frac{1}{q}-\frac{1}{4}},\frac{3}{4}}$.
Theorem~\ref{thm:embedding} implies
\eqref{el:w}  for $r_0=r$ 
and $\frac{3}{q_0}\in\big(\max\set{0,\frac{3}{q}-2},\frac{3}{q}\big]$.
H\"older's inequality now yields
\eqref{el:wgradv}
for $q_4$ and $r_4$ as in iii.
Additionally, if $q> 3$,
then we obtain  \eqref{el:reg.ss1} by Lemma~\ref{lem:wgradw}.
Theorem \ref{thm:embedding} further implies
\eqref{el:w} for $r_0=r$ and $q_0\in[q,\infty]$,
so that H\"older's inequality yields
\eqref{el:wgradv}
for $q_4$ and $r_4$ as in iv.

For v., we distinguish two cases.
Firstly, if $q\leq 3$, then 
Theorem~\ref{thm:embedding} implies
\eqref{el:w}  for $\frac{3}{q_0}=\frac{3}{q}-1-\delta$ 
and $\frac{2}{r_0}\in\big(\max\set{0,\frac{2}{r}-1+\delta},2\big]$ for $\delta\in(0,1)$,
and Lemma~\ref{lem:reg.initial} yields 
\eqref{el:gradv} for all $s_1=\frac{3}{1+\delta}$,
so that H\"older's inequality implies
\eqref{el:wgradv} for $q_4=q_0$ and $r_4=r_0$.
Secondly, if $q>3$, then we use Lemma~\ref{lem:wgradw} again
to conclude \eqref{el:reg.ss1}.
Moreover, Theorem~\ref{thm:embedding} yields
\eqref{el:w} for $q_0=q$ and $r_0\in[1,\infty]$,
and we conclude
\eqref{el:wgradv} for $q_4=q$ and $r_4=r_0\in[1,\infty]$.
Combining both cases, we obtain v.
\end{proof}

The results from Lemma~\ref{lem:wgradw} and Lemma~\ref{lem:vgradwwgradv}
are not sufficient to conclude the proof,
and we need to invoke the additional regularity assumptions 
\eqref{el:reg.fct} or \eqref{el:reg.grad}
to get \eqref{el:wgradw} for other parameters $q_2$ and $r_2$.
We define $\delta_{\kappa,\rho}>0$ by
\[
\delta_{\kappa,\rho}
\coloneqq
\begin{cases}
1-s_{\kappa,\rho} & \text{if \eqref{el:reg.fct} is assumed},
\\
2-s_{\kappa,\rho} & \text{if \eqref{el:reg.grad} is assumed}.
\end{cases}
\]

\begin{Lemma}\label{lem:wgradw.betterreg}
Assume either \eqref{el:reg.fct} or \eqref{el:reg.grad},
and let $\wvel\in\calw_{q,r}$ for some $q,r\in(1,\infty)$.
Then
\eqref{el:wgradw} holds for
\begin{enumerate}
\item[i.]
$\frac{3}{q_2}\in\big(\max\setl{\frac{3}{\kappa},\frac{3}{q}-\delta_{\kappa,\rho}},\min\setl{3,\frac{3}{q}+\frac{3}{\kappa}}\big]$
and $r_2=r$, and
\item[ii.]
$q_2=q$ and
$\frac{2}{r_2}\in\big(\max\setl{\frac{2}{\rho},\frac{2}{r}-\delta_{\kappa,\rho}},2\big]$.
\end{enumerate}
\end{Lemma}

\begin{proof}
At first, let us assume \eqref{el:reg.fct}.
Then $\rho>2$
and Theorem~\ref{thm:embedding}
with $\beta=\frac{2}{\rho}\in(0,1)$
yields \eqref{el:gradw} 
for 
$\frac{3}{q_1}\in
\big(\max\setl{0,\frac{3}{q}+s_{\kappa,\rho}-1-\frac{3}{\kappa}},\frac{3}{q}\big]$
and $\frac{2}{r_1}=\frac{2}{r}-\frac{2}{\rho}$.
Combining this with \eqref{el:reg.fct}
and using H\"older's inequality,
we obtain \eqref{el:wgradw}
for $q$, $r$ as in i.
Moreover, 
we have
$\kappa>3$,
and Theorem~\ref{thm:embedding}
with $\beta=1-\frac{3}{\kappa}\in(0,1)$
yields \eqref{el:gradw} 
for 
$\frac{2}{r_1}\in
\big(\max\setl{0,\frac{2}{r}+s_{\kappa,\rho}-1-\frac{2}{\rho}},2\big]$
and $\frac{3}{q_1}=\frac{3}{q}-\frac{3}{\kappa}$.
Combining this with \eqref{el:reg.fct}
and using H\"older's inequality,
we obtain \eqref{el:wgradw}
for $q$, $r$ as in ii.

Now let us assume \eqref{el:reg.grad}.
From
Theorem~\ref{thm:embedding}
with $\alpha=\frac{2}{\rho}\in(0,2)$
we deduce \eqref{el:w} 
for
$\frac{3}{q_0}\in
\big(\max\setl{0,\frac{3}{q}+s_{\kappa,\rho}-2-\frac{3}{\kappa}},\frac{3}{q}\big]$
and $\frac{2}{r_0}=\frac{2}{r}-\frac{2}{\rho}$.
Combining this with \eqref{el:reg.grad}
and using H\"older's inequality,
we also obtain \eqref{el:wgradw} in this case
for $q$, $r$ as in i.
Moreover,
we have $\kappa>\frac{3}{2}$,
and
Theorem~\ref{thm:embedding}
with $\alpha=2-\frac{3}{\kappa}\in(0,2)$
yields \eqref{el:w} 
for
$\frac{2}{r_0}\in
\big(\max\setl{0,\frac{2}{r}+s_{\kappa,\rho}-2-\frac{2}{\rho}},2\big]$
and $\frac{3}{q_0}=\frac{3}{q}-\frac{3}{\kappa}$.
Combining this with \eqref{el:reg.grad}
and using H\"older's inequality,
we also obtain \eqref{el:wgradw} in this case
for $q$, $r$ as claimed in ii.
\end{proof}

Now we have prepared everything to iteratively 
increase the range of parameters $q$, $r$ such that
$\wvel\in\calw_{q,r}$.
By Lemma \ref{lem:reg.initial},
we start with $q$, $r$ such that $s_{q,r}=4$. 
In particular, both parameters cannot be chosen 
large,
and we use Lemma~\ref{lem:vgradwwgradv}
and Lemma~\ref{lem:wgradw.betterreg}
to extend the range of admissible parameters.
An iteration leads to sufficiently large parameters such that Lemma~\ref{lem:wgradw} can be invoked
to further iterate until the full range $(1,\infty)$ 
is admissible for both parameters,
which proves the regularity result from
Theorem~\ref{thm:embedding}.

\begin{proof}[Proof of Theorem~\ref{thm:regularity}]

As a first step, we show that 
$\wvel\in\calw_{q,r}$ for all $q\in(3,\infty)$ 
and all $r\in(2,\infty)$.
To do so, 
observe that both \eqref{el:reg.fct} and \eqref{el:reg.grad} imply
that $\kappa>3$ or $\rho>2$.
In what follows, we distinguish these two cases:

Consider the case $\kappa>3$ at first.
We show that
$\wvel\in\calw_{\tilde q, r}$ for all $\tilde q\in(1,\kappa)$ and $r\in(1,2)$.
Let $q\in(1,\frac{3}{2})$ and $r\in(1,2)$ 
with $s_{q,r}=4$,
so that $\wvel\in\calw_{q,r}$
by Lemma~\ref{lem:reg.initial}.
Then we have \eqref{el:wgradw}
for $q_2$, $r_2$ 
as in i.~of Lemma~\ref{lem:wgradw.betterreg},
we have \eqref{el:vgradw} for $q_3$, $r_3$
as in i.~of Lemma~\ref{lem:vgradwwgradv},
and we have \eqref{el:wgradv} for $q_4$, $r_4$
as in iii.~of Lemma~\ref{lem:vgradwwgradv}.
We thus obtain
\eqref{el:rhs.tp}
for
$\frac{3}{q_5}\in\bp{\max\setl{\frac{3}{\kappa},\frac{3}{q}-\delta_{\kappa,\rho},\frac{3}{q}-\frac{3}{4},\frac{3}{4},\frac{6}{q}-\frac{11}{4}},\min\setl{3,\frac{3}{q}+\frac{3}{\kappa},\frac{3}{q}+\frac{3}{2}}}$
and $r_5=r$.
Since $\kappa>\frac{3}{2}$, this interval is non-empty, 
and by the regularity result from Lemma~\ref{lem:OseenReg}, we conclude
$\wvel\in\calw_{\tilde q,r}$ for 
$\frac{3}{\tilde q}\in\bp{\max\setl{\frac{3}{\kappa},\frac{3}{q}-\delta_{\kappa,\rho},\frac{3}{q}-\frac{3}{4},\frac{3}{4},\frac{6}{q}-\frac{11}{4}},\min\setl{3,\frac{3}{q}+\frac{3}{\kappa},\frac{3}{q}+\frac{3}{2}}}$.
Repeating this argument iteratively with $q$ replaced with a suitable $\tilde q$ within this range,
we obtain
$\wvel\in\calw_{\tilde q,r}$ for $\tilde q\in(1,\min\set{4,\kappa})$.
If $\kappa\leq 4$, this completes the first step. 
If this is not the case,
we repeat the above argument for $q\in(3,4)$, 
but we use iv.~of Lemma~\ref{lem:vgradwwgradv} instead of iii.,
which leads to
\eqref{el:rhs.tp}
for
$\frac{3}{q_5}\in\bp{\max\setl{\frac{3}{\kappa},\frac{3}{q}-\delta_{\kappa,\rho}},\min\setl{3,\frac{3}{q}+\frac{3}{\kappa},\frac{3}{q}+\frac{3}{2}}}$
and $r_5=r$,
and thus $\wvel\in\calw_{\tilde q,r}$ for 
for
$\frac{3}{\tilde q}\in\bp{\max\setl{\frac{3}{\kappa},\frac{3}{q}-\delta_{\kappa,\rho}},\min\setl{3,\frac{3}{q}+\frac{3}{\kappa},\frac{3}{q}+\frac{3}{2}}}$.
Repeating now this argument a sufficient number of times for admissible $\tilde q>q$ instead of $q$,
we finally arrive at
$\wvel\in\calw_{\tilde q, r}$ for all $\tilde q\in(1,\kappa)$ and $r\in(1,2)$.

Since we assume $\kappa>3$, we can now choose $q\in(3,\kappa)$.
Let $r\in(1,2)$ such that $s_{q,r}<2$.
The previous step implies $\wvel\in\calw_{q, r}$,
and we conclude
\eqref{el:wgradw}
for $q_2$, $r_2$ 
as in ii.~of Lemma~\ref{lem:wgradw},
we have \eqref{el:vgradw} for $q_3$, $r_3$
as in ii.~of Lemma~\ref{lem:vgradwwgradv},
and we have \eqref{el:wgradv} for $q_4$, $r_4$
as in v.~of Lemma~\ref{lem:vgradwwgradv}.
We thus obtain
\eqref{el:rhs.tp}
for $q_5=q$ and 
$\frac{2}{r_5}
\in\big(\max\setl{0,s_{q,r}-1},2\big]$.
Invoking Lemma~\ref{lem:OseenReg},
we obtain
$\wvel\in\calw_{q,\tilde r}$ for
$\frac{2}{\tilde r}
\in\bp{\max\setl{0,\frac{2}{r}+\frac{3}{q}-1},2}$,
and an iteration as above yields 
$\wvel\in\calw_{\tilde q,\tilde r}$ for all $\tilde q\in(3,\kappa)$ 
and all $\tilde r\in(1,\infty)$.

Now let $q\in(3,\kappa)$ and $r\in(2,\infty)$.
Then $s_{q,r}<2$ and
since $\wvel\in\calw_{q,r}$,
we have \eqref{el:wgradw}
for $q_2$, $r_2$ 
as in i.~of Lemma~\ref{lem:wgradw},
we have \eqref{el:vgradw} for $q_3$, $r_3$
as in i.~of Lemma~\ref{lem:vgradwwgradv},
and we have \eqref{el:wgradv} for $q_4$, $r_4$
as in iv.~of Lemma~\ref{lem:vgradwwgradv}.
We thus obtain
\eqref{el:rhs.tp}
for
$\frac{3}{q_5}\in\bp{\max\setl{0,s_{q,r}-1,\frac{3}{q}+s_{q,r}-2},\frac{6}{q}}$
and $r_5=r$,
and Lemma~\ref{lem:OseenReg}
yields
$\wvel\in\calw_{\tilde q,r}$ for
$\frac{3}{\tilde q}\in\bp{\max\setl{0,\frac{3}{q}+\frac{2}{r}-1,\frac{3}{q}+s_{q,r}-2},\frac{6}{q}}$.
An iteration of this argument
leads to $\wvel\in\calw_{\tilde q,\tilde r}$ for
all $\tilde q\in(\frac{3}{2},\infty)$ and all $\tilde r\in(2,\infty)$.

Now consider the case $\rho>2$.
We first extend the range for $r$ and show that 
$\wvel\in\calw_{\tilde q,\tilde r}$ 
for all $\tilde q\in(1,\frac{3}{2})$ and $\tilde r\in(1,\rho)$.
For this, fix $q\in(1,\frac{3}{2})$.
Lemma~\ref{lem:reg.initial}
yields $\wvel\in\calw_{q,r}$ for some $r\in(1,2)$
such that $s_{q,r}=4$.
Then we have \eqref{el:wgradw}
for $q_2$, $r_2$ 
as in ii.~of Lemma~\ref{lem:wgradw.betterreg},
we have \eqref{el:vgradw} for $q_3$, $r_3$
as in ii.~of Lemma~\ref{lem:vgradwwgradv},
and we have \eqref{el:wgradv} for $q_4$, $r_4$
as in v.~of Lemma~\ref{lem:vgradwwgradv}.
We thus obtain
\eqref{el:rhs.tp}
for $q_5=q$ and 
$\frac{2}{r_5}
\in\big(\max\setl{\frac{2}{\rho},\frac{2}{r}-\delta_{\kappa,\rho},\frac{2}{r}-1},2\big]$.
Using the regularity result from Lemma \ref{lem:OseenReg},
we now obtain
$\wvel\in\calw_{q,\tilde r}$ for
$\frac{2}{\tilde r}
\in\bp{\max\setl{\frac{2}{\rho},\frac{2}{r}-\delta_{\kappa,\rho},\frac{2}{r}-1},2}$.
Repeating this argument with $r$ replaced with some $\tilde r>r$ in this range,
we can successively increase the admissible range
for $\tilde r$ until we obtain 
$\wvel\in\calw_{\tilde q,\tilde r}$ for all $\tilde q\in(1,\frac{3}{2})$ and $\tilde r\in(1,\rho)$.

Since $\rho>2$, we can choose $r\in(2,\rho)$, 
and from $\wvel\in\calw_{q,r}$
for $q\in(1,\frac{3}{2})$
and we conclude 
\eqref{el:wgradw}
for $q_2$, $r_2$ 
as in i.~of Lemma~\ref{lem:wgradw},
we have \eqref{el:vgradw} for $q_3$, $r_3$
as in i.~of Lemma~\ref{lem:vgradwwgradv},
and we have \eqref{el:wgradv} for $q_4$, $r_4$
as in iii.~of Lemma~\ref{lem:vgradwwgradv}.
We thus obtain
\eqref{el:rhs.tp}
for $r_5=r$ and 
$\frac{3}{q_5}
\in\big(\max\setl{\frac{3}{4},\frac{3}{q}-\frac{3}{4},\frac{6}{q}-\frac{11}{4},s_{q,r}-1, \frac{3}{q}+s_{q,r}-2},\min\set{3,\frac{6}{q}}\big]$.
Invoking Lemma~\ref{lem:OseenReg},
we obtain
$\wvel\in\calw_{\tilde q, r}$ for
$\frac{3}{\tilde q}
\in\bp{\max\setl{\frac{3}{4},\frac{3}{q}-\frac{3}{4},\frac{6}{q}-\frac{11}{4},s_{q,r}-1, \frac{3}{q}+s_{q,r}-2},\min\set{3,\frac{6}{q}}}$,
and an iteration as above yields 
$\wvel\in\calw_{\tilde q,r}$ for all $\tilde q\in(1,4)$.
Now we can choose $q=\tilde q>3$,
and repeating the argument with
iv.~of Lemma~\ref{lem:vgradwwgradv}
instead of iii., we obtain
$\wvel\in\calw_{\tilde q, r}$ for
$\frac{3}{\tilde q}
\in\bp{\max\setl{0,s_{q,r}-1, \frac{3}{q}+s_{q,r}-2},\min\set{3,\frac{6}{q}}}$.
Another iteration now leads to
$\wvel\in\calw_{\tilde q,\tilde r}$ for all $\tilde q\in(1,\infty)$ 
and all $\tilde r\in(2,\rho)$
if $\rho>2$.

Now let $q\in(3,\infty)$ and $r\in(2,\rho)$.
Then $s_{q,r}<2$ and
since $\wvel\in\calw_{q,r}$,
we have \eqref{el:wgradw}
for $q_2$, $r_2$ 
as in ii.~of Lemma~\ref{lem:wgradw},
we have \eqref{el:vgradw} for $q_3$, $r_3$
as in ii.~of Lemma~\ref{lem:vgradwwgradv},
and we have \eqref{el:wgradv} for $q_4$, $r_4$
as in v.~of Lemma~\ref{lem:vgradwwgradv}.
We thus obtain
\eqref{el:rhs.tp}
for $q_5=q$ and 
$\frac{2}{r_5}
\in\big(\max\setl{0,s_{q,r}-1},2\big]$.
Invoking Lemma~\ref{lem:OseenReg},
we obtain
$\wvel\in\calw_{q,\tilde r}$ for
$\frac{2}{\tilde r}
\in\bp{\max\setl{0,\frac{2}{r}+\frac{3}{q}-1},2}$,
and an iteration as above yields 
$\wvel\in\calw_{\tilde q,\tilde r}$ for all $\tilde q\in(3,\infty)$ 
and all $\tilde r\in(1,\infty)$.

Combining these two cases and using that $\torus$ is compact,
we have shown that
$\wvel\in\calw_{q,r}$ for all $q\in(3,\infty)$ and
$r\in(1,\infty)$. 
In particular, 
$\vvel$ satisfies~\eqref{el:reg.ss0} and \eqref{el:reg.ss1}
by Lemma~\ref{lem:wgradw},
and we have \eqref{el:reg.ss2}
by Lemma~\ref{lem:reg.initial}.
To conclude~\eqref{el:reg.pp},
note that 
Theorem~\ref{thm:embedding} implies
\eqref{el:w} and \eqref{el:gradw} 
for $q_0,q_1\in(3,\infty]$ and $r_0,r_1\in[1,\infty)$ ,
so that \eqref{el:wgradw} holds for $q_2\in(\frac{3}{2},\infty]$
and $r_2\in(1,\infty)$,
and i.~and iv.~of Lemma~\ref{lem:vgradwwgradv} yield
\eqref{el:vgradw} for
$q_3\in(\frac{6}{5},\infty)$ and $r_3\in(1,\infty)$,
and \eqref{el:wgradv} for $q_4\in(1,\infty)$ and
$r_4\in(1,\infty)$.
We thus have 
obtain
\eqref{el:rhs.tp}
for $q_5\in(\frac{3}{2},\infty)$ and $r_5\in(1,\infty)$,
and from Lemma~\ref{lem:OseenReg} we conclude
$\wvel\in\calw_{q,r}$ for all 
$q\in(\frac{3}{2},\infty)$ and $r\in(1,\infty)$. 
Repeating the argument once more, 
we finally obtain~\eqref{el:reg.pp}.
Moreover, \eqref{el:reg.pres}
is a direct consequence of 
\eqref{el:reg.ss2}--\eqref{el:reg.pp}
in virtue of
\eqref{sys:NavierStokesTP_sspart} 
and~\eqref{sys:NavierStokesTP_pppart}.

Finally, \eqref{el:reg:ss2.full}
follows from Theorem~\ref{thm:statNS} and the additional assumptions 
on $\partial\Omega$, $f$ and $\uvel_\ast$
since
\eqref{el:reg.pp}
implies that $\proj(\wvel\cdot\grad\wvel)\in\WSR{1}{q}(\Omega)$
for any $q\in(1,\infty)$.
\end{proof}

\begin{proof}[Proof of Theorem~\ref{thm:smooth}]
At first, we increase the time regularity of the solution 
inductively in steps of half a derivative. 
For $j\in\N$
let $\tuvel_j\coloneqq\sqrt{D}_t^j\uvel$
and $\twvel_j\coloneqq\sqrt{D}_t^j\wvel$.
We show that 
for every $j\in\N$
we have
\begin{equation}
\label{el:tuvel.reg0}
\begin{aligned}
\forall q,r\in(1,\infty): \ \quad
\twvel_j\in \WSR{1}{r}(\torus;\LR{q}(\Omega)^3)\cap\LR{r}(\torus;\WSR{2}{q}(\Omega)^3).
\end{aligned}
\end{equation}
By Theorem~\ref{thm:regularity}
there exists a pressure field $\upres$
such that $(\uvel,\upres)$ is a strong solution to~\eqref{sys:NavierStokesTP}
with the regularity stated in \eqref{el:reg.ss2}--\eqref{el:reg:ss2.full}.
In particular, this shows
\eqref{el:tuvel.reg0} for $j=0$.
Now assume that $\twvel_j$ has the asserted regularity
stated in \eqref{el:tuvel.reg0} for all $j\in\set{0,\ldots,k}$.
Then Theorem~\ref{thm:embedding} implies
\begin{equation}\label{el:tuvel.reg1}
\forall q,r\in(1,\infty):
\quad
\twvel_{j+1}
=\sqrt{D}_t \twvel_j
\in\LR{r}(\torus;\WSR{1}{q}(\Omega)^3)
\cap\LR{\infty}(\torus;\LR{\infty}(\Omega)^3)
\end{equation}
for $j=0,\dots,k$.
Let $\varphi\in \CRcisigma(\torus\times\Omega)$ 
and multiply \eqref{sys:NavierStokesTP.a}
by $\sqrt{D}_t^{k+1} \varphi$.
Since $\twvel_k=\tuvel_k$ for $k\geq1$,
after integrating by parts in space and time
as well as by means of~\eqref{eq:fracder.ibp},
we obtain 
\begin{equation}
\label{eq:weakform.pdt}
\begin{aligned}
\int_\torus\int_\Omega\bb{-\twvel_{k+1}\cdot\partial_t\varphi
+\grad\twvel_{k+1}:\grad\varphi
-\rey\partial_1\twvel_{k+1}\cdot\varphi
&+(\twvel_{k+1}\cdot\grad\twvel_{k+1})\cdot\varphi}\,\dx\dt
\\
&\qquad\quad
=\int_\torus\int_\Omega f_{k+1}
\cdot\varphi\,\dx\dt,
\end{aligned}
\end{equation}
where
\[
f_{k+1}\coloneqq 
\sqrt{D}_t^{k+1} f
+\twvel_{k+1}\cdot\grad\twvel_{k+1}
-\sqrt{D}_t^{k+1}\Div\np{\uvel\otimes\uvel}.
\]
In virtue of the smoothness of the boundary 
data and the regularity of $\twvel$,
we see that $\twvel$ is 
a weak solution to the Navier--Stokes equations~\eqref{sys:NavierStokesTP}
for the right-hand side $f_{k+1}$,
which is an element of $\LR{r}(\torus;\LR{q}(\Omega)^3)$
for all $q,r\in(1,\infty)$.
For the first two terms 
in the definition of $f_{k+1}$, 
this follows from the assumptions and 
from~\eqref{el:tuvel.reg1}.
For the term $\sqrt{D}_t^{k+1}\Div\np{\uvel\otimes\uvel}$
we distinguish two cases.

If $k=2N-1$ is an odd number, then
this term is an element of $\LR{r}(\torus;\LR{q}(\Omega))$
if and only if $\partial_i\partial_t^N\np{\uvel\otimes\uvel}$
is an element of $\LR{r}(\torus;\LR{q}(\Omega))$
for $i=1,2,3$.
We write
\[
\begin{aligned}
\partial_i\partial_t^N\np{\uvel\otimes\uvel}
=\sum_{\ell=0}^{N}\partial_i\partial_t^\ell\uvel\otimes\partial_t^{N-\ell}\uvel
\end{aligned}
\]
We can estimate the terms of this sum as
\[
\begin{aligned}
\norm{\partial_i\uvel\otimes\partial_t^{N}\uvel}_{\LR{r}(\torus;\LR{q}(\Omega))}
&\leq
\norm{\grad\uvel}_{\LR{\infty}(\torus;\LR{\infty}(\Omega))}
\norm{\twvel_{k+1}}_{\LR{r}(\torus;\LR{q}(\Omega))},
\\
\norm{\partial_i\partial_t^N\uvel\otimes\uvel}_{\LR{r}(\torus;\LR{q}(\Omega))}
&\leq
\norm{\grad\twvel_{k+1}}_{\LR{r}(\torus;\LR{q}(\Omega))}
\norm{\uvel}_{\LR{\infty}(\torus;\LR{\infty}(\Omega))},
\\
\norm{\partial_i\partial_t^\ell\uvel\otimes\partial_t^{N-\ell}\uvel}_{\LR{r}(\torus;\LR{q}(\Omega))}
&\leq
\norm{\grad\twvel_{2\ell}}_{\LR{r}(\torus;\LR{q}(\Omega))}
\norm{\twvel_{k+1-2\ell}}_{\LR{\infty}(\torus;\LR{\infty}(\Omega))},
\end{aligned}
\]
for $\ell=1,\dots,N-1$,
where the respective right-hand side is finite 
due to~\eqref{el:reg.ss1}, \eqref{el:reg.ss0},
\eqref{el:reg.pp} and the embedding Theorem~\ref{thm:embedding}
as well as 
\eqref{el:tuvel.reg1} for $j\leq k$.
If $k=2N$ is an even number, 
then $\sqrt{D}_t^{k+1}\Div\np{\uvel\otimes\uvel}\in\LR{r}(\torus;\LR{q}(\Omega)^3)$
if and only if this is true for
\[
\sqrt{D}_t\partial_t^N\Div\np{\uvel\otimes\uvel}
=\sum_{\ell=0}^N
\sqrt{D}_t\Div\bp{\partial_t^{\ell}\uvel\otimes\partial_t^{N-\ell}\uvel}.
\]
By Theorem~\ref{thm:embedding},
this is the case
if 
$\partial_t^{\ell}\uvel\otimes\partial_t^{N-\ell}\uvel
\in\WSR{1}{r}(\torus;\LR{q}(\Omega))\cap
\LR{r}(\torus;\WSR{2}{q}(\Omega))$
for $\ell=0,\dots,N$.
For example, for the terms with derivatives of highest order
we obtain
\[
\begin{aligned}
\norm{\uvel\otimes\partial_t^{N+1}\uvel}_{\LR{r}(\torus;\LR{q}(\Omega))}
&\leq
\norm{\uvel}_{\LR{\infty}(\torus;\LR{\infty}(\Omega))}
\norm{\partial_t\twvel_{k}}_{\LR{r}(\torus;\LR{q}(\Omega))},
\\
\norm{\partial_t^{\ell}\uvel\otimes\partial_t^{N+1-\ell}\uvel}_{\LR{r}(\torus;\LR{q}(\Omega))}
&\leq
\norm{\twvel_{2\ell}}_{\LR{r}(\torus;\LR{q}(\Omega))}
\norm{\partial_t\twvel_{k-2\ell}}_{\LR{\infty}(\torus;\LR{\infty}(\Omega))},
\\
\norm{\uvel\otimes\partial_t^{N}\grad^2\uvel}_{\LR{r}(\torus;\LR{q}(\Omega))}
&\leq
\norm{\uvel}_{\LR{\infty}(\torus;\LR{\infty}(\Omega))}
\norm{\grad^2\twvel_{k}}_{\LR{r}(\torus;\LR{q}(\Omega))},
\\
\norm{\partial_t^{N}\uvel\otimes\grad^2\uvel}_{\LR{r}(\torus;\LR{q}(\Omega))}
&\leq
\norm{\twvel_{k}}_{\LR{\infty}(\torus;\LR{\infty}(\Omega))}
\norm{\grad^2\uvel}_{\LR{r}(\torus;\LR{q}(\Omega))},
\\
\norm{\partial_t^{\ell}\uvel\otimes\partial_t^{N-\ell}\grad^2\uvel}_{\LR{r}(\torus;\LR{q}(\Omega))}
&\leq
\norm{\twvel_{2\ell}}_{\LR{r}(\torus;\LR{q}(\Omega))}
\norm{\grad^2\twvel_{k-2\ell}}_{\LR{\infty}(\torus;\LR{\infty}(\Omega))},
\end{aligned}
\]
which are all finite by the same argument as above.
Similarly, this follows for the lower-order terms.

In summary, we obtain $f_{k+1}\in\LR{r}(\torus;\LR{q}(\Omega)^3)$
for all $q,r\in(1,\infty)$ in both cases.
By \eqref{el:tuvel.reg1},
the function $\twvel_{k+1}$ is subject to both regularity 
assumptions \eqref{el:reg.fct} and \eqref{el:reg.grad},
and Theorem~\ref{thm:regularity}
implies
that $\twvel_{k+1}=\projcompl\twvel_{k+1}$
satisfies \eqref{el:tuvel.reg0}
for $j=k+1$.
We thus have shown \eqref{el:tuvel.reg0} for all~$j\in\N_0$.

To increase the spatial regularity,
we recall that $\np{\uvel,\upres}$ is a strong solution by 
Theorem~\ref{thm:regularity},
so that the $N$-th time derivative, $N\in\N_0$,
satisfies the Stokes system
\[
\begin{aligned}
-\Delta\partial_t^N\uvel+\grad\partial_t^N\upres
&= F_N
&&\qquad\tin\Omega, 
\\
\Div\partial_t^N\uvel&=0 
&&\qquad\tin\Omega, \\
\partial_t^N\uvel&=\partial_t^N\uvel_\ast 
&&\qquad\ton\partial\Omega
\end{aligned}
\]
a.e.~in $\torus$,
where 
\[
F_N\coloneqq
\partial_t^N f
-\partial_t^{N+1}\uvel+\rey\partial_1\partial_t^N\uvel-
\partial_t^N\bp{\uvel\cdot\grad\uvel}.
\]
Since $\partial_t^\ell\uvel=\partial_t^\ell\wvel$ for $\ell\geq 1$,
Theorem~\ref{thm:regularity}
and~\eqref{el:tuvel.reg0}
imply
$F_N\in\LR{r}(\torus;\WSR{1}{q}(\Omega_R)^3)$
for all $q,r\in(1,\infty)$ and all $R>0$ such that 
$\partial\Omega\subset\ball_R$,
where we define $\Omega_R\coloneqq\Omega\cap\ball_R$,
and $\ball_R\subset\R^3$ is the ball with radius $R$ and centered at $0\in\R^3$.
By a classical regularity result for the steady-state Stokes problem 
(see \cite[Theorem IV.5.1]{GaldiBookNew} for example),
we obtain
$\partial_t^N\uvel\in\LR{r}(\torus;\WSR{3}{q}(\Omega_R))$
for all $R>0$ sufficiently large and all $N\in\N_0$.
This implies $F_N\in\LR{r}(\torus;\WSR{2}{q}(\Omega_R))$,
and can again apply~\cite[Theorem IV.5.1]{GaldiBookNew}
to deduce $\partial_t^N\uvel\in\LR{r}(\torus;\WSR{4}{q}(\Omega_R))$.
Iterating this argument, we finally obtain
\[
\uvel\in\WSR{N}{r}(\torus;\WSR{M}{q}(\Omega_R))
\]
for all $N,M\in\N_0$, all $q,r\in(1,\infty)$ and all $R>0$ such that 
$\partial\Omega\subset\ball_R$.
This completes the proof.
\end{proof}

\smallskip\par\noindent
Weierstrass Institute for Applied Analysis and Stochastics\\ 
Mohrenstra\ss{}e 39, 10117 Berlin, Germany\\
Email: {\texttt{thomas.eiter@wias-berlin.de}}

\end{document}